\documentclass{amsart}
\usepackage{amssymb}
\usepackage{graphicx}
\usepackage[shortlabels]{enumitem}
\usepackage{multirow}
\usepackage{enumitem}
\usepackage[T1]{fontenc}
\usepackage{amsmath,color}
\usepackage{hyperref}
\usepackage{url}
\usepackage{longtable}

\theoremstyle{plain}
\newtheorem{theorem}{Theorem}[section]

\newtheorem{proposition}[theorem]{Proposition}
\newtheorem{problem}[theorem]{Problem}
\newtheorem{lemma}[theorem]{Lemma}

\theoremstyle{definition}

\theoremstyle{remark}

{

\title{Hypercubes are determined by their distance spectra}

\author[J. H. Koolen]{ Jack H. Koolen$^\diamondsuit$}
\thanks{$^\diamondsuit$J.H.K. is partially supported by the National Natural Science Foundation of China (No. 11471009).}
\author[S. Hayat]{ Sakander Hayat$^\spadesuit$}
\thanks{$^\spadesuit$S.H is supported by a CAS-TWAS president's fellowship at USTC, China.}
\author[Q. Iqbal]{ Quaid Iqbal$^\clubsuit$}
\thanks{$^\clubsuit$Q.I is supported by a Chinese government fellowship at USTC, China.}

\address{Wen-Tsun Wu Key Laboratory of CAS, School of Mathematical Sciences,
University of Science and Technology of China,
Hefei, Anhui, 230026, P.R. China}
\email{koolen@ustc.edu.cn}

\address{School of Mathematical Sciences,
University of Science and Technology of China,
Hefei, Anhui, 230026, P.R. China}
\email{sakander1566@gmail.com}

\address{School of Mathematical Sciences,
University of Science and Technology of China,
Hefei, Anhui, 230026, P.R. China}
\email{quaid.iqbal@yahoo.com}

\begin{document}



\maketitle


\begin{abstract}
We show that the $d$-cube is determined by the spectrum of its distance matrix.
\end{abstract}
\section{Introduction}
For undefined terms, see next section.
For integers $n\geq2$ and $d\geq2$, the \emph{Hamming graph} $H(d,n)$ has as vertex set, the $d$-tuples with elements
from $\{0,1,2,\ldots,n-1\}$, with two $d$-tuples are adjacent if and only if they differ only in one coordinate.
For a positive integer $d$, the $d$-\emph{cube} is the graph $H(d,2)$.
In this paper, we show that the $d$-cube is determined by its distance spectrum, see Theorem \ref{main-thm}.
Observe that the $d$-cube has exactly three distinct distance eigenvalues.\\

Note that the $d$-cube is not determined by its adjacency spectrum as the Hoffman graph \cite{H63} has the same
adjacency spectrum as the $4$-cube and hence the cartesian product of $(d-4)$-cube and the Hoffman graph has the
same adjacency spectrum as the $d$-cube (for $d\geq4$, where the $0$-cube is just $K_{1}$). This also implies that
the complement of the $d$-cube is not determined by its distance spectrum, if $d\geq4$.\\

There are quite a few papers on matrices of graphs with a few distinct eigenvalues. An important case
is, when the Seidel matrix of a graph has exactly two distinct eigenvalues is very much studied, see, for example \cite{S76}.
Van Dam and Haemers \cite{VH98} characterized the graphs whose Laplacian matrix has two distinct nonzero eigenvalues.
The non-regular graphs with three distinct adjacency eigenvalues studied in Bridges \& Mena \cite{BM81},
Muzychuk \& Klin \cite{MK98}, Van Dam \cite{V98}, and Cheng et al. \cite{CGGK15,CGK15}. In this paper, we focus
on connected graphs with three distinct distance eigenvalues.\\

Also the question whether a graph $\Gamma$ is determined by its spectrum of a matrix $M=M(\Gamma)$, where $M$
is the adjacency matrix, the Laplacian matrix, the signless Laplacian matrix and so on, has received much attention, see,
for example, Van Dam and Haemers \cite{VH06,VH09}.
The study of the distance spectrum of a connected graph has obtained a considerable attention in the last few years, see
Aouchiche \& Hansen \cite{AH014}, for a survey paper on this topic. The study of question whether a graph is determined
by its distance spectrum is just in its beginning. See McKay \cite{M77}, Lin et al. \cite{LHWS013}, and Jin \& Zhang \cite{JZ014} for
recent papers on this subject.\\

This paper is organized as follows. In Section \ref{sec2}, we will give definitions and preliminaries. In Section \ref{sec3},
we will give some old and new results on the second largest and smallest distance eigenvalue.
We also will give an alternative proof for the fact that the complete multipartite graphs are determined by
their distance spectrum, a result that was first shown by Jin \& Zhang \cite{JZ014}. In Section \ref{sec3.1}, we
will look at connected graphs with exactly three distinct distance eigenvalues
and develop some basic theory for them. Prop. \ref{prop1} is crucial for the proof of our main result. In
Section \ref{sec4}, we will give a proof of the fact that the $d$-cube is determined by its distance spectrum,
our main result. In Section \ref{sec5}, we will give some open problems.

\section{Preliminaries}\label{sec2}
All the graphs in this paper are simple and undirected. A \emph{graph} $\Gamma$ is a pair $(V(\Gamma),E(\Gamma))$,
where $V(\Gamma)$ is a finite set and $E(\Gamma)\subseteq{V(\Gamma)\choose2}$. The set $V(\Gamma)$ is the
vertex set of $\Gamma$ and $E(\Gamma)$ is the edge set of $\Gamma$. We define $n_{\Gamma}=\#V(\Gamma)$ and
$\epsilon_{\Gamma}=\#E(\Gamma)$. We denote the matrix of all-ones, the identity matrix and  the vector of all-ones
by $J$, $I$ and $\textbf{j}$ respectively.\\

Let $\Gamma$ be a connected graph. The \emph{distance} between two vertices $x$ and $y$,
$d_{\Gamma}(x,y)$, is defined as the length of the shortest walk between $x$ and $y$. Note that the space $(V(\Gamma),d_{\Gamma})$
is a \emph{metric space}. The \emph{diameter} of a graph $\Gamma$ is the maximum distance between any two vertices of $\Gamma$.
The \emph{distance matrix} $\mathcal{D}(\Gamma)$ is defined as the
symmetric $(n\times n)$-matrix indexed by the $V(\Gamma)$ such that $\mathcal{D}(\Gamma)_{xy}=d_{\Gamma}(x,y)$.
When it is clear from the context which graph $\Gamma$ we mean, we delete $\Gamma$ from the
notations like $n_{\Gamma}$, $\epsilon_{\Gamma}$, $d_{\Gamma}(x,y)$, and $\mathcal{D}(\Gamma)$. A
\emph{distance eigenvalue} of $\Gamma$ is an eigenvalue of its distance matrix. As the distance matrix
of a graph is a symmetric integral matrix, it has an orthonormal basis of eigenvectors corresponding to $n$ real eigenvalues.
In particular, the algebraic and geometric multiplicities of its eigenvalues coincide.
We denote these eigenvalues as $\delta_{1}\geq\delta_{2}\geq\ldots\geq\delta_{n}$, if we list them all, and as
$\delta_{0}^{(m_{0})}>\delta_{1}^{(m_{1})}>\ldots>\delta_{t}^{(m_{t})}$, if we list the distinct ones, where $m_i$ is the multiplicity
of $\delta_{i}$.\\

Let $\Gamma$ be a connected graph, with distinct distance eigenvalues $\delta_{0}>\ldots>\delta_{t}$.
By the Perron-Frobenius theorem (see Brouwer \& Haemers \cite[Theorem 2.2.1]{BH12}), we have $\delta_{0}\geq\mid\delta_{t}\mid$,
and $\delta_{0}$ has multiplicity 1. We call $\delta_{0}$ the \emph{distance spectral radius}.
The \emph{distance spectrum} of $\Gamma$ is the spectrum of its distance matrix.
The characteristic polynomial of $\mathcal{D}$ of a graph $\Gamma$ is
called its \emph{distance characteristic polynomial} denoted as $p_{\Gamma}(x)$. As
$\mathcal{D}$ only contains integral entries, the polynomial $p_{\Gamma}(x)$ is a monic polynomial
with integral coefficients. Let $\delta$ be an eigenvalue of $\mathcal{D}$, say with multiplicity $m$, and
let $q(x)$ be the minimal polynomial of $\delta$. It is well-known that $q(x)^{m}$ divides $p_{\Gamma}(x)$. It follows that
for any other root $\delta^{\prime}$ of $q(x)$ we have, that $\delta^{\prime}$ is an  eigenvalue of $\mathcal{D}$ with multiplicity
$m$. So, for example, if any other eigenvalue of $\mathcal{D}$ has multiplicity at least $2$, then $\delta_{0}$ is integral.\\

We say that a connected graph $\Omega$ is \emph{distance cospectral} with $\Gamma$ if the two distance spectra are the
same, and that a connected graph $\Gamma$ is \emph{determined by its distance spectrum} if
there does not exist a non-isomorphic graph $\Omega$ that is distance cospectral with $\Gamma$.\\

In this paper, by interlacing we always mean the following `interlacing theorem'.
\begin{lemma}\emph{(\cite[Theorem 9.3.3]{gr01})}\label{lem:interlacing}
Let $A$ be a Hermitian matrix of order $n$, and let $B$ be a principal submatrix of $A$
of order $m$. If $\theta_{1}(A)\geq\theta_{2}(A)\geq\ldots\geq\theta_{n}(A)$ lists the eigenvalues of $A$
and $\mu_{1}(B)\geq\mu_{2}(B)\geq\ldots\geq\mu_{m}(B)$ the eigenvalues of $B$, then
$\theta_{n-m+i}(A)\leq\mu_{i}(B)\leq\theta_{i}(A)$ for $i=1,2,\ldots,m$.
\end{lemma}
We refer the reader to Brouwer \& Haemers \cite[Section 2.3]{BH12}, for
the necessary background on equitable partitions of real symmetric matrices and Godsil \& Royle
\cite[Chapter 9]{gr01} for background on interlacing.\\

Suppose $M$ is a symmetric real matrix whose rows and columns are indexed by
$X =\{1,\ldots,n\}$. Let $\{X_1,\ldots,X_m\}$ be a partition of $X$. The \emph{characteristic matrix}
$P$ is the $n\times m$ matrix whose $j$-th column is the characteristic vector of $X_j$
$(j=1,\ldots,m)$. Define $n_i=\mid X_i\mid$ and $K=\mathtt{diag}(n_1,\ldots,n_m)$. Let $M$ be partitioned according
to $\{X_1,\ldots,X_m\}$, that is,
$$M=\left(
              \begin{array}{ccc}
                M_{1,1} & \ldots & M_{1,m}\\
                \vdots & \ddots & \vdots\\
                M_{m,1} & \ldots & M_{m,m}\\
              \end{array}
           \right),$$
where $M_{i,j}$ denotes the submatrix (block) of $M$ whose rows (resp. columns) are indexed to be the
elements in $X_i$ (resp. $X_{j}$).
Let $q_{i,j}$ denote the average row sum of $M_{i,j}$. Then the matrix $Q=(q_{i,j})$ is
called the \emph{quotient matrix} of $M$ with respect to the given partition. We have
\begin{equation*}
KQ=P^{\top}AP,\hspace{0.2cm}P^{\top}P=K.
\end{equation*}
If the row sum of each block $M_{i,j}$ is constant then the partition is called \emph{equitable}
with respect to $M$,
and we have $M_{i,j}\textbf{j}=b_{i,j}\textbf{j}$ for $i,j=0,\ldots,d$, so
\begin{equation*}
AP=PQ.
\end{equation*}
The following lemma gives a result on equitable partitions of real symmetric matrices.
The proof is a straightforward modification of the proof of \cite[Theorem 9.1.1]{gr01}.
\begin{lemma}\emph{(cf. \cite[Theorem 9.1.1]{gr01})}
Let $M$ be a real symmetric matrix and let $\pi$ be an equitable partition with respect to $M$ with quotient matrix $Q$.
Then the characteristic polynomial of the quotient matrix $Q$ divides the characteristic polynomial of $M$.
\end{lemma}

\section{Bounds on distance eigenvalues}\label{sec3}
This section is devoted to the study of the distance spectrum of complete multipartite graphs.
Item (2) of the following proposition was shown by Lin et al. \cite[Theorem 2.3 and Remark 2.4]{LHWS013}.
Item (3) is a slight extension of Lin et al. \cite[Theorem 2.6]{LHWS013}.
We give a full proof of the proposition for the convenience of the reader.
\begin{proposition}\label{prop3}
Let $\Gamma$ be a connected graph with $n\geq2$ vertices, distance matrix $\mathcal{D}$ and diameter $D$.
Let $\delta_{1}$ and $\delta_{\min}$ be respectively the second largest and smallest distance eigenvalues,
with respective multiplicities $m_{1}$ and $m_{\min}$. Then the following all hold.
\begin{itemize}
 \item[\emph{(1)}] $-\delta_{\min}\geq D$ holds and equality can only occur if $-\delta_{\min}\in\{1,2\}$.\label{prp31}
 \item[\emph{(2)}] $\delta_{\min}>-2$ if and only if $\delta_{\min}=-1$ if and only if $\Gamma$ is the complete graph $K_{n}$.\label{prp32}
 \item[\emph{(3)}] $\delta_{\min}=-2$ if and only if $\Gamma$ is a complete $s$-partite graph $K_{a_1,a_2,\ldots,a_{s}}$,
where $1\leq a_{1}\leq\ldots\leq a_{s}$ are integers such that $a_{s}\geq2$ and $s=n-m_{\min}$.\label{prp33}
\item[\emph{(4)}] We have $n-1\leq\delta_{0}\leq D(n-1)$ and $\delta_{0}=D(n-1)$ if and only if $D=1$.
\item[\emph{(5)}] The following statements are equivalent;
\begin{itemize}
\item[\emph{(i)}] $\delta_{1}<1-\sqrt{3}$;
\item[\emph{(ii)}] $\delta_{1}=-1$;
\item[\emph{(iii)}] $\Gamma$ is the complete graph $K_{n}$;
\item[\emph{(iv)}] $\delta_{0}=n-1$.
\end{itemize}\label{prp34}
\end{itemize}
\end{proposition}
\begin{proof}
Let $x,y$ and $z$ be the distinct vertices of $\Gamma$, such that $d(x,y)=1$, $d(y,z)=t-1$, and $d(x,z)=t$, and $t\geq2$.
The principal submatrix of $\mathcal{D}$ induced by $\{x,y,z\}$ equals
          $\left(
              \begin{array}{ccc}
                0 & 1 & t \\
                1 & 0 & t-1\\
                t & t-1 & 0\\
              \end{array}
            \right).$
This matrix has smallest eigenvalue at most $-t$ and equality can only hold if $t=2$. This shows (1) by interlacing.

For (2), we obtain that if $\delta_{\min}>-2$, then $D<2$ holds by (1), and hence $\Gamma$ is complete.
This shows (2).

For (3), we find that if $\delta_{\min}=-2$, then $D=2$ holds by (1) and (2). The pentagon and the triangle with a pendant edge
have smallest distance eigenvalue less than $-2$, so by interlacing they can not be induced subgraphs of $\Gamma$. It follows that,
to be at distance $2$ or $0$ is an equivalence relation, which implies that $\Gamma$ is complete multipartite.

Now, let $\Gamma=K_{a_1,a_2,\ldots,a_{s}}$, for some integers $a_{i}\geq1$ and $s\geq1$.
Let $\pi=\{V_{i}\mid i=1,2,\ldots,s\}$ be the color classes of $\Gamma$ such that $a_{i}=\#V_{i}$. Then $\pi$ is equitable
partition with respect to the distance matrix of $\Gamma$ with the quotient matrix
$$B=\left(
              \begin{array}{cccc}
                2a_{1}-2 & a_{2} & \ldots & a_{s}\\
                a_{1} & 2a_{2}-2 &  \ldots & a_{s}\\
                a_{1} &  a_{2} & \ldots & a_{s}\\
                \vdots & \vdots & \ddots & \vdots\\
                a_{1} & a_{2} & \ldots & 2a_{s}-2\\
              \end{array}
           \right).$$
It follows that $\det{(B+2I)}=a_1a_2\ldots a_{s}$. By induction on $s$, it is easy to show that $B+2I$
has only positive eigenvalues, so $\delta_{\min}=-2$. It also follows that the rank of $B+2I$ is equal to $s$.
Let $\chi_{i}$ be the characteristic vector of $V_{i}$ for $i=1,2,\ldots,s$. This means that all the eigenvectors
of $\mathcal{D}_{\Gamma}$ with respect to $-2$ are orthogonal to $\chi_{i}$ for all $i=1,2,\ldots,s$.
This means that $-2$, as a distance eigenvalue, has the same multiplicity as $0$ for the adjacency matrix. This shows that
$s=n-m_{\min}$. This shows (3).

For (4), let $\mathbf{x}$ be the Perron-Frobenius eigenvector for $\mathcal{D}$, then $(J-I)\mathbf{x}\leq\mathcal{D}\mathbf{x}\leq D(J-I)\mathbf{x}$,
and $(J-I)\mathbf{x}=\mathcal{D}\mathbf{x}$ if and only if $D=1$ if and only if $D(J-I)\mathbf{x}=\mathcal{D}\mathbf{x}$.
This shows (4).

For (5), note that the second largest distance eigenvalue of $K_{1,2}$ is equal to $1-\sqrt{3}$.
So this means that if $\delta_{1}<1-\sqrt{3}$, the graph is complete. Using (4), we see that (5) holds.
\end{proof}

Now we give the following result due to Jin \& Zhang \cite{JZ014}. We present an alternative proof for it.
\begin{theorem}\label{com-multi-DDS}
Let $\Gamma=K_{a_1,a_2\ldots,a_s}$ be the complete $s$-partite graph with $\sum\limits_{i=1}^sa_{i}=n$. Then $\Gamma$ is determined
by its distance spectrum.
\end{theorem}
\begin{proof}
Let $\delta_{1}\geq\delta_{2}\geq\ldots\geq\delta_{n}$ be the distance eigenvalues of $K_{a_1,a_2\ldots,a_s}$
with $a_1\geq a_2\geq\ldots\geq a_{s}\geq1$. Let $\Omega$ be the distance cospectral graph with $\Gamma$.
Then $\Omega$ has smallest distance eigenvalue $-2$ with multiplicity $n-s$. So by part $(3)$ of Prop. \ref{prop3},
the graph $\Omega$ is a $K_{b_1,b_2\ldots,b_s}$ with $\sum\limits_{i=1}^sb_{i}=n$ and $b_1\geq b_2\geq\ldots\geq b_{s}\geq1$.
We need to show that $b_{i}=a_{i}$ for $i=1,2,\ldots,s$.

Note that $\delta_{1}+2,\ldots,\delta_{s}+2$ are the eigenvalues of the both matrices\\
$B^{\prime}=\left(
              \begin{array}{cccc}
                2a_{1} & a_{2} & \ldots & a_{s}\\
                a_{1} & 2a_{2} &  \ldots & a_{s}\\
                a_{1} &  a_{2} & \ldots & a_{s}\\
                \vdots & \vdots & \ddots & \vdots\\
                a_{1} & a_{2} & \ldots & 2a_{s}\\
              \end{array}
\right)$
and
$C^{\prime}=\left(
              \begin{array}{cccc}
                2b_{1} & b_{2} & \ldots & b_{s}\\
                b_{1} & 2b_{2} &  \ldots & b_{s}\\
                b_{1} &  b_{2} & \ldots & b_{s}\\
                \vdots & \vdots & \ddots & \vdots\\
                b_{1} & b_{2} & \ldots & 2b_{s}\\
              \end{array}
\right).$
Consider the characteristic polynomial $\chi(B^{\prime})$ of $B^{\prime}$ and let
$c_{i}$ be the coefficient of $x^{i}$ in $\chi(B^{\prime})$.
Then $c_{i}=\gamma_{i}\sum\limits_{\sigma\in Sym(n)}\prod\limits_{j=1}^{s-i}a_{\sigma(j)}$
for some $\gamma_{i}$ not depending on the $a_{j}$'s and $\sigma\in Sym(n)$.
Note that $\gamma_{i}\neq0$, as $B^{\prime}$ (and $C^{\prime}$) have only positive eigenvalues, and the $\gamma_{i}$'s
do not depend on the $a_{i}$'s. That means that for all $i=1,2,\ldots,s$,
\begin{equation*}
\sum\limits_{\sigma\in Sym(n)}\prod_{j=1}^{i}a_{\sigma(j)}=\sum\limits_{\sigma\in Sym(n)}\prod_{j=1}^{i}b_{\sigma(j)}.
\end{equation*}
This implies that $a_{i}=b_{i}$ for all $i$ as $\prod\limits_{i=1}^{s}(x-a_{i})=\prod\limits_{i=1}^{s}(x-b_{i})$.
This shows the theorem.
\end{proof}

\section{Graphs with Three Distance Eigenvalues}\label{sec3.1}
An important property of connected graphs with three distance eigenvalues is that
$(\mathcal{D}-\delta_{1}I)(\mathcal{D}-\delta_{2}I)$ is a rank one matrix.
Since the largest eigenvalue $\delta_{0}$ is simple, by the Perron-Frobenius Theorem, it follows that we have
\begin{equation}\label{eq2}
(\mathcal{D}-\delta_{1}I)(\mathcal{D}-\delta_{2}I)=\mathbf{\alpha}\mathbf{\alpha}^{\top},
\hspace{0.5cm}\textrm{with}\hspace{0.3cm}\mathcal{D}\mathbf{\alpha}=\delta_{0}\mathbf{\alpha}
\end{equation}

\begin{equation}\label{eq2*}
\mathcal{D}^{2}=\alpha\alpha^{\top}+(\delta_{1}+\delta_{2})\mathcal{D}+\delta_{1}\delta_{2}I.
\end{equation}

The following result was shown by Lin et al. \cite{LHWS013} for diameter two. We prove that the result is true for any diameter.
\begin{lemma}\label{com.bip}
Let $\Gamma$ be an $n$-vertex connected graph with exactly three distinct distance eigenvalues, say $\delta_{0}>\delta_{1}>\delta_{2}$.
If $\delta_{1}$ or $\delta_{2}$ is simple, then $\Gamma$ is complete bipartite. In particular if $\delta_{0}\notin\mathbb{Z}$, then
$\Gamma$ is complete bipartite.
\end{lemma}
\begin{proof}
Let $\mathcal{D}$ be the distance matrix of $\Gamma$.
For $n=3$, the lemma is obviously true, so we consider $n\geq4$. Let $\{\delta,\delta^{\prime}\}=\{\delta_{1},\delta_{2}\}$
and let $\delta$ be a simple distance eigenvalue.
Then $\delta^{\prime}\in\mathbb{Z}$.
As the trace of $\mathcal{D}$ is equal to zero, we find
\begin{equation*}
\delta_{0}+\delta=-\delta^{\prime}(n-2).
\end{equation*}
The rank of $\mathcal{D}-\delta^{\prime}I$ is equal to $2$. Clearly $\Gamma$ is not complete, so consider vertices $x,y,z$ such that
$x\sim y\sim z$ and $d(x,z)=2$. The principle submatrix $P$ of $\mathcal{D}$ with respect to $\{x,y,z\}$ satisfies
$P=\left(
              \begin{array}{ccc}
                0 & 1 & 2 \\
                1 & 0 &  1\\
                2 &  1 &  0\\
              \end{array}
            \right)$
and $P$ has eigenvalues $1\pm\sqrt{3},-2$. As $P-\delta^{\prime}I$ has rank at most $2$, we see that $\delta^{\prime}\in\{1\pm\sqrt{3},-2\}$.
But as $\delta^{\prime}\in\mathbb{Z}$, we see $\delta^{\prime}=-2$. Thus we are done by Prop. \ref{prop3}.
\end{proof}

The following result shows that the graphs with exactly three distinct distance eigenvalues fall into three classes.
\begin{proposition}\label{3-class}
Let $\Gamma$ be an $n$-vertex connected graph with exactly three distinct distance eigenvalues $\delta_{0}>\delta_{1}>\delta_{2}$,
with respective multiplicities $m_{0}=1,m_{1},m_{2}$. Then one of the following holds.
\begin{itemize}
  \item[\emph{(1)}] $\Gamma$ is complete bipartite;
  \item[\emph{(2)}] $n$ is odd and $\delta_{0}=c\big(\frac{n-1}{2}\big)$ with $3\leq c\in\mathbb{Z}$, and $m_{1}=m_{2}\geq2$;
  \item[\emph{(3)}] $\delta_{0},\delta_{1},\delta_{2}\in\mathbb{Z}$, $\delta_{1}\geq0$ and $\delta_{2}\leq-3$.
\end{itemize}
\end{proposition}
\begin{proof}
If $\delta_{1}$ or $\delta_{2}$ is simple then by Lemma \ref{com.bip}, we find (1).
If $m_{1}=m_{2}\geq2$ then $m_{1}=m_{2}=m=\frac{1}{2}(n-1)$ and hence $n=2m+1$
is an odd integer.
Since the trace of distance matrix is zero, consequently the sum of (distance) eigenvalues is zero, and hence we obtain
\begin{equation}\label{eq1}
\delta_{0}+\frac{1}{2}(n-1)(\delta_{1}+\delta_{2})=0.
\end{equation}
As $0>\delta_{1}+\delta_{2}\in\mathbb{Z}$, we obtain $\delta_{0}=c\frac{1}{2}(n-1)$, where $c\in\mathbb{Z}$,
by Equation (\ref{eq1}). As $\Gamma$ is not a complete graph, we find $c>2$ by Prop. \ref{prp34}(5).
Hence, if $m_{1}=m_{2}\geq2$, then $\Gamma$ falls to (2). So we may assume $m_{1}$, $m_{2}$ both are at least 2 and $m_{1}\neq m_{2}$.
Then the (distance) eigenvalues are integral. Thus we are done by Prop. \ref{prop3}(1).
\end{proof}
For any two vertices $x,y\in V(\Gamma)$, we define the notations $\nu_{ij}(x,y)$ and $k_{i}(x)$, where
$\nu_{ij}(x,y)=\#\{z\in V(\Gamma)\mid d(x,z)=i \hspace{0.2cm}\textrm{and}\hspace{0.2cm} d(y,z)=j\}$ and
$k_{i}(x)=\#\{z\in V(\Gamma)\mid d(x,z)=i\}$. Note that if $y=x$ and $j=i$, then $\nu_{ij}(x,y)=k_{i}(x)$.
The following result is true for a connected graph with three distinct distance eigenvalues.
\begin{proposition}\label{prop1}
Let $\Gamma$ be a connected graph with diameter $D$ and with exactly three distinct distance eigenvalues $\delta_{0}>\delta_{1}>\delta_{2}$.
There exists an eigenvector $\alpha$ of $\mathcal{D}$ corresponding to $\delta_{0}$ such that for any two distinct vertices $x,y\in V(\Gamma)$, we have
\begin{equation*}
\sum\limits_{0\leq i,j\leq D}(i-j)^{2}\nu_{ij}=-2(\delta_{1}+\delta_{2})d(x,y)-2\delta_{1}\delta_{2}+(\alpha_{x}-\alpha_{y})^{2}.
\end{equation*}
\end{proposition}
\begin{proof}
For any two distinct vertices $x,y\in V(\Gamma)$ with $d(x,y)=t$ and using Equation (\ref{eq2*}), we obtain the following three equations:
\begin{equation}\label{eq3}
\sum\limits_{i=0}^{D}i^{2}k_{i}(x)=\alpha_{x}^{2}-\delta_{1}\delta_{2},
\end{equation}
\begin{equation}\label{eq4}
\sum\limits_{i=0}^{D}i^{2}k_{i}(y)=\alpha_{y}^{2}-\delta_{1}\delta_{2},
\end{equation}
\begin{equation}\label{eq5}
\sum\limits_{\substack{0\leq i,j\leq D \\ 0\leq\mid i-j\mid\leq t}}ij\nu_{ij}=
\alpha_{x}\alpha_{y}+(\delta_{1}+\delta_{2})t,
\end{equation}
where $t=d(x,y)$ and $\nu_{ij}=\nu_{ij}(x,y)$ as defined above. Now adding Equations (\ref{eq3}) and (\ref{eq4})
and subtracting twice Equation (\ref{eq5}), we obtain the desired result.
\end{proof}

\section{The Main Result}\label{sec4}
In this section, we present our main result. We will prove that the $d$-cube
is determined by its distance spectrum.\\
The distance spectrum of the $d$-cube is given in the following lemma.
\begin{lemma}\emph{(\cite{AA015})}\label{lem:dcube}
Let $d$ be a positive integer. The distance spectrum of the $d$-cube is
given as,
\begin{equation*}
[d2^{d-1}]^{1}, [0]^{(2^{d}-d-1)}, [-2^{d-1}]^{d}
\end{equation*}
\end{lemma}

To show that the $d$-cube is determine by its distance spectrum, we will need some theory
about isometrically embeddable graphs in a hypercube.
We say that a connected graph $\Gamma$ is \emph{isometrically embeddable} in a connected
graph $\Omega$ if there exists a map $\phi:V(\Gamma)\rightarrow V(\Omega)$ such that
$d_{\Gamma}(x,y)=d_{\Omega}(\phi(x),\phi(y))$ for all vertices $x$ and $y$ of $\Gamma$.
We will say that in this case $\Gamma$ is an isometric subgraph
of $\Omega$.\\

The following result shows that for a connected bipartite graph, the property that
the distance matrix has only one positive eigenvalue is very strong. The equivalence
of Items $(1)$ and $(3)$ is due to Roth \& Winkler \cite{RW86}, and the fact that Items $(1)$ and
$(2)$ are equivalent is due to Graham \& Winkler \cite{GW85}. Note that the distance matrix
of an $d$-cube has rank $d+1$, see Lemma \ref{lem:dcube}.
\begin{theorem}\emph{(\cite[Theorem 19.2.8]{DL97}, \cite{GW85} and \cite{RW86})}\label{thm1}
Let $\Gamma$ be a connected
bipartite graph with at least two vertices. Let $m$ be the rank of its distance matrix
$\mathcal{D}$. The following assertions are equivalent.
\begin{itemize}
  \item[\emph{(1)}] $\Gamma$ is isometrically embeddable in a $d$-cube for some positive integer $d$.
  \item[\emph{(2)}] $\Gamma$ is isometrically embeddable in the $(m-1)$-cube.
  \item[\emph{(3)}] The distance matrix of $\Gamma$ has exactly one positive eigenvalue.
\end{itemize}
\end{theorem}
We now prove a result for bipartite graphs whose distance matrix has constant row sum
and three distinct distance eigenvalues. This result plays an important role in the proof of our main theorem.
\begin{proposition}\label{prop2}
Let $\Gamma$ be a connected bipartite graph with exactly three distinct distance eigenvalues, say $\delta_{0}>\delta_{1}>\delta_{2}$.
Let $\mathcal{D}_{\Gamma}$ be the distance matrix of $\Gamma$ and $D$ is the diameter and
$\mathcal{D}_{\Gamma}\mathbf{j}=\delta_{0}\mathbf{j}$. Let $\Omega$ be a connected graph
that is distance-cospectral with $\Gamma$. Let $\mathcal{D}_{\Omega}$ be the distance matrix of $\Omega$.
Then $\mathcal{D}_{\Omega}\mathbf{j}=\delta_{0}\mathbf{j}$ and $\Omega$ is bipartite.
\end{proposition}
\begin{proof}
Let $x,y$ be adjacent vertices in $\Gamma$. Let $\alpha$ be such that
\begin{equation*}\label{eqI}
(\mathcal{D}_{\Gamma}-\delta_{1}I)(\mathcal{D}_{\Gamma}-\delta_{2}I)=\mathbf{\alpha}\mathbf{\alpha}^{\top},
\hspace{0.5cm}\textrm{with}\hspace{0.3cm}\mathcal{D}_{\Gamma}\mathbf{\alpha}=\delta_{0}\mathbf{\alpha}.
\end{equation*}
Then, as $\Gamma$ is bipartite, and by Prop. \ref{prop1}, we have
\begin{eqnarray*}
n=\sum\limits_{i=0}^{D-1}\nu_{i,i+1}^{\Gamma}(x,y)+\nu_{i+1,i}^{\Gamma}(x,y)&=&\sum\limits_{0\leq i,j\leq D}(i-j)^{2}\nu_{ij}^{\Gamma}(x,y)\\
&=&-2(\delta_{1}+\delta_{2}+\delta_{1}\delta_{2})+(\alpha_{x}-\alpha_{y})^{2},
\end{eqnarray*}
where $\nu_{ij}^{\Gamma}(x,y)$ is $\nu_{ij}(x,y)$ in the graph $\Gamma$ as defined above.

As $\mathcal{D}_{\Gamma}\mathbf{j}=\delta_{0}\mathbf{j}$ holds, we have $\alpha_{x}=\alpha_{y}$ for all vertices $x$ and $y$ of $\Gamma$,
and therefore,
\begin{equation}\label{eqII}
-2(\delta_{1}+\delta_{2}+\delta_{1}\delta_{2})=n.
\end{equation}
Let $\tilde{x},\tilde{y}$ be adjacent vertices in $\Omega$.
Let $\mathbf{\tilde{\alpha}}$ be such that
\begin{equation*}
(\mathcal{D}_{\Omega}-\delta_{1}I)(\mathcal{D}_{\Omega}-\delta_{2}I)=\mathbf{\tilde{\alpha}}\mathbf{\tilde{\alpha}}^{\top},
\hspace{0.5cm}\textrm{with}\hspace{0.3cm}\mathcal{D}_{\Omega}\mathbf{\tilde{\alpha}}=\delta_{0}\mathbf{\tilde{\alpha}}.
\end{equation*}
Then, by Prop. \ref{prop1}, we find
\begin{eqnarray}\label{eqIII}
\sum\limits_{i=0}^{D-1}\nu_{i,i+1}^{\Omega}(x,y)+\nu_{i+1,i}^{\Omega}(x,y)&=&\sum\limits_{0\leq i,j\leq D}(i-j)^{2}\nu_{ij}^{\Omega}(\tilde{x},\tilde{y})\\
\nonumber &=&-2(\delta_{1}+\delta_{2}+\delta_{1}\delta_{2})+(\mathbf{\tilde{\alpha}}_{\tilde{x}}-\mathbf{\tilde{\alpha}}_{\tilde{y}})^{2}.
\end{eqnarray}
Thus, by Equations (\ref{eqII}) and (\ref{eqIII}), we see that
\begin{equation*}
\sum\limits_{i=0}^{D-1}\nu_{i,i+1}^{\Omega}(\tilde{x},\tilde{y})+\nu_{i+1,i}^{\Omega}(\tilde{x},\tilde{y})\geq n,
\end{equation*}
holds. As $\sum\limits_{i=0}^{D-1}\nu_{i,i+1}^{\Omega}(\tilde{x},\tilde{y})+\nu_{i+1,i}^{\Omega}(\tilde{x},\tilde{y})\leq n$ clearly holds, so we find
\begin{equation*}
\sum\limits_{i=0}^{D-1}\nu_{i,i+1}^{\Omega}(\tilde{x},\tilde{y})+\nu_{i+1,i}^{\Omega}(\tilde{x},\tilde{y})= n,\hspace{0.5cm}\textrm{and}\hspace{0.3cm}\mathbf{\tilde{\alpha}}_{\tilde{x}}=\mathbf{\tilde{\alpha}}_{\tilde{y}}.
\end{equation*}
The first equation implies that $\Omega$ is bipartite. As $\Omega$ is connected, we obtain $\mathbf{\tilde{\alpha}}_{\tilde{x}}=\mathbf{\tilde{\alpha}}_{\tilde{z}}$
for all vertices $\tilde{x}$ and $\tilde{z}$ of $\Omega$, and hence $\mathcal{D}_{\Omega}\mathbf{j}=\delta_{0}\mathbf{j}$.
\end{proof}
Now we present our main theorem of this paper.
\begin{theorem}\label{main-thm}
Let $\Gamma=Q_{d}$. Then $\Gamma$ is determined by the spectrum of its distance matrix.
\end{theorem}
\begin{proof}
Observe that the distance spectrum of $\Gamma$ is $[d\times2^{d-1}]^{1}, [0]^{(2^{d}-d-1)}, [-2^{d-1}]^{d}$, see Lemma \ref{lem:dcube}.
Let $\mathcal{D}_{\Gamma}$ be the distance spectrum of $\Gamma$. Observe that $\mathcal{D}_{\Gamma}\mathbf{j}=d2^{d-1}\mathbf{j}$.\\
Let $\Omega$ be a connected graph with distance matrix $\mathcal{D}_{\Omega}$ and distance spectrum\\
$[d\times2^{d-1}]^{1}, [0]^{(2^{d}-d-1)}, [-2^{d-1}]^{d}$. Then by Prop. \ref{prop2}, $\Omega$ is bipartite, as
$\Gamma$ is bipartite, and $\mathcal{D}_{\Gamma}\mathbf{j}=d2^{d-1}\mathbf{j}$. The matrix $\mathcal{D}_{\Omega}$
has one positive eigenvalue and the rank of $\mathcal{D}_{\Omega}$ is equal to $d+1$.
By Theorem \ref{thm1}, we obtain that $\Omega$ must be isometrically embeddable in $\Gamma$.
But, as $\Omega$ has the same number of vertices as $\Gamma$, the graph $\Omega$ is isomorphic to $\Gamma$.
\end{proof}
\section{Open Problems}\label{sec5}
We finish this paper with some open problems.
\begin{problem}\label{prob1}
Determine the connected graphs with three distinct distance eigenvalues $\delta_{0}>\delta_{1}>\delta_{2}$
such that $\delta_{1}=0$. Examples are the Johnson graphs, Hamming graphs and the halved cubes. When the class
of graphs is restricted to distance-regular graphs, then this has been solved by Aalipour et al. \cite{AA015}
using the main result of Koolen \& Shpectorov \cite{KS94}.
\end{problem}

\begin{problem}\label{prob2}
Determine the connected graphs with three distinct distance eigenvalues $\delta_{0}>\delta_{1}>\delta_{2}$
such that $\delta_{2}=-3$. Examples are the strongly-regular graphs with second largest adjacency eigenvalue $1$.
\end{problem}

\begin{problem}\label{prob3}
Find a pair of distance cospectral graphs $\Gamma$ and $\Omega$ with respective distance matrices $\mathcal{D}_{\Gamma}$ and $\mathcal{D}_{\Omega}$ such that $\mathcal{D}_{\Gamma}\mathbf{j}=\delta_{0}\mathbf{j}$ and $\mathcal{D}_{\Omega}\mathbf{j}\neq\delta_{0}\mathbf{j}$, if they do exist.
\end{problem}

\begin{problem}\label{prob4}
Does the distance spectrum of a graph determine whether a graph is bipartite?
\end{problem}

\begin{problem}\label{prob5}
Does there exist a non-distance-regular graph with the same distance spectrum as a distance-regular graph?
There are many non-distance-regular graphs with the same adjacency spectrum as a distance-regular graph,
see Van Dam et al. \cite{VHKS06}.
\end{problem}

As a special case of Problem \ref{prob5}, we would like to propose the following problem:
\begin{problem}\label{prob6}
Are the Hamming graphs $H(d,q)$ determined  by their distance spectrum if $q\neq4$, (as for $q=4$, we have the Doob graphs).
Note that, some work has been done for the Hamming graph with respect to the adjacency spectrum by Bang et al. \cite{BVK008}.
\end{problem}
\vspace{0.4cm}
\noindent\textbf{Acknowledgements:}\\

We would like to thank the referee for his/her comments which improved the paper significantly.

\end{document}